\documentclass[11pt,letterpaper]{amsart}

\usepackage{amsfonts,amsthm,amsmath,amsfonts,latexsym,amssymb}

\usepackage{graphicx,upref,calligra,float}
\usepackage[svgnames]{xcolor}
\usepackage{mathrsfs}

\usepackage[T1]{fontenc}

\usepackage[applemac]{inputenc}
\usepackage[svgnames]{xcolor}
\usepackage{colortbl}
\usepackage[colorlinks=true,linkcolor=azul,citecolor=azul]{hyperref}

\theoremstyle{plain}
\newtheorem{thm}{Theorem}[section]
\newtheorem{lem}[thm]{Lemma}
\newtheorem{cor}[thm]{Corollary}
\newtheorem{prop}[thm]{Proposition}
\theoremstyle{definition}
\newtheorem{rem}[thm]{Remark}

\newtheorem{exas}[thm]{Examples}
\numberwithin{equation}{section}

\definecolor{azul}{rgb}{0.1,0.6,0.86}
\definecolor{titleblue}{rgb}{0.13,0.49,0.69}
\definecolor{mylred}{rgb}{0.85,0.24,0.2}
\definecolor{myblue}{rgb}{0,0.33,0.55}
\definecolor{myyellow}{rgb}{0.42,0.24,0.52}
\definecolor{mygreen}{rgb}{0.12,0.5,0.29}
\definecolor{myred}{rgb}{0.74,0.13,0.13}
\definecolor{mylblue}{rgb}{0.2,0.75,1}
\definecolor{mylgreen}{rgb}{0.68,0.98,0.6}
\definecolor{mylyellow}{rgb}{0.86,0.85,0.55}
\definecolor{myllyellow}{rgb}{0.87,0.86,0.56}
\definecolor{naranja}{RGB}{249,153,96}
\definecolor{sidebardarkcolor}{rgb}{0.21,0.31,0.40}
\definecolor{sidebarlightcolor}{rgb}{0.7,0.77,0.836}

\def\bdem{\begin{proof}}
\def\edem{\end{proof}}
\def\ds{\displaystyle}

\newcommand{\sub}[2]{{#1}_{\mbox{\tiny{${#2}$}}}}

\def\la{\lambda}

\def\W{\Omega}

\def\N{\mathbb{N}}

\def\R{\mathbb{R}}

\def\C{\mathbb{C}}

\def\H{\mathbb{H}}

\def\1{\mathds{1}}

\def\cC{\mathcal{C}}

\def\H{\mathcal{H}}
\def\cH{\mathcal{H}}

\DeclareMathOperator\arctanh{arctanh}
\DeclareMathOperator\Vol{Vol}

\newcommand{\pint}[1]{\displaystyle \left \langle\, #1 \, \right\rangle}

\newcommand{\bm}[1]{\mathcal{B}_{#1}}

\newcommand{\isob}[1]{\mathcal{I}(\bm{#1})}

\begin{document}

\title{Weighted maximal inequalities on hyperbolic spaces } 
	
	\author{Jorge Antezana}
	\address{Departamento de Matem\'atica, Universidad Aut\'onoma de Madrid; Centro de Matem\'atica de La Plata, FCE-UNLP; Ins\-ti\-tu\-to Argentino de Matem\'atica ``Alberto P. Calder\'on'' - CONICET.
		}
	\email{jorge.antezana@uam.es}
	
	\author{Sheldy Ombrosi}
	\address{Departamento de An\'alisis Matem\'atico y Matem\'atica Aplicada\\ Universidad Complutense, Spain \&
Departamento de Matem\'atica e Instituto de Matem\'atica. Universidad Nacional del Sur - CONICET, Argentina}
	\email{sombrosi@ucm.es}

\date{}

\begin{abstract} 

 In this work we develop a weight theory in the setting of hyperbolic spaces. Our starting point is a variant of the well-known endpoint Fefferman-Stein inequality for the centered Hardy-Littlewood maximal function. This inequality generalizes, in the hyperbolic setting, the weak $(1,1)$ estimates obtained by Str\"omberg in \cite{Str} who answered a question posed by Stein and Wainger in \cite{SW}. Our approach is based on a combination of geometrical arguments and the techniques used in the discrete setting of regular trees by Naor and Tao in \cite{NT}. This variant of the Fefferman-Stein inequality paves the road to weighted estimates for the maximal function for $p>1$. On the one hand, we show that the classical $A_p$ conditions are not the right ones in this setting. 
On the other hand, we provide sharp sufficient conditions for weighted weak and strong type $(p,p)$ boundedness of the centered maximal function, when $p>1$. The sharpness is in the sense that, given $p>1$, we can construct a weight satisfying our sufficient condition for that $p$, and so it satisfies the weak type $(p,p)$ inequality, but the strong type $(p,p)$ inequality fails. In particular, the weak type $(q,q)$ fails as well for every $q < p$. 
\end{abstract}

\maketitle

\baselineskip=17pt
	\let\thefootnote\relax\footnote{ 
            
		
		J. Antezana was supported by grants: PICT 2019 0460 (ANPCyT), 
		PIP112202101 00954CO (CONICET),
		11X829 (UNLP),
		PID2020-113048GB-I00 (MCI).
		S. Ombrosi was supported by PID2020-113048GB-I00 (MCI).
	}


\section{Introduction}

Let $\H^n$ denote the $n$-dimensional hyperbolic space, i.e. the unique (up to isometries) $n$-dimensional, complete, and simply connected Riemannian manifold with constant sectional curvature $-1$. Let $\mu_n$ denote the corresponding volume measure. If $\sub{B}{H}(x,r)$ denotes the hyperbolic ball of radio $r$ centered at $x$, then the centered Hardy{-}Littlewood maximal function on $\H^n$
is defined as 
\[
Mf(x)=\sup_{r>0}\frac{1}{\mu_{n}(B_{H}(x,r))}\int_{B_{H}(x,r)}|f(y)|d\mu_{n}(y).
\]

In the seminal work  \cite{SW}, Stein and Wainger proposed the study of the end-point estimates for the centered Hardy-Littlewood maximal function when the curvature of the underline space could be non-negative. In this more general scenario, the euclidean spaces $\R^n$ and the aforementioned hyperbolic spaces $\H^n$ represent two extreme cases. 

\medskip

In \cite{Str}, Str\"omberg proved the (unweighted) weak type $(1,1)$ boundedness of $M$ in symmetric spaces of noncompact type, suggesting that the behavior of the maximal operator is the same in both spaces, $\R^n$ and $\H^n$. However, this is not the case in general, and it will be reveled by analyzing weighted estimates. More precisely, to complete the answer to Stein-Wainger's question we study  an end-point two-weight Fefferman-Stein inequality for $M$ in the hyperbolic setting.

\subsection{Fefferman Stein type inequality}
In the Euclidean setting, the classical Feffer\-man Stein inequality \cite{FS} is
$$
   w\left(\left\{ x\in \R^n \,:\,Mf(x)>\lambda\right\} \right)\lesssim \frac{1}{\lambda} \int_{\R^n}|f(x)|\,Mw(x)dx,
$$
where  $w$ is  non-negative measurable function (a weight) defined in $\R^n$, and  $w(E)=\int_E w(x) dx$. This is a cornerstone in the theory of weights, and a powerful tool to consider vector valued extension of the maximal function $M$. This result follows from a classical covering lemma,  which is not available in the hyperbolic setting. Indeed, in this setting
\begin{equation}\label{lanza la bola chico}
\mu_n \Big(\sub{B}{H}(x,r)\Big)= \Omega_n \int_0^r (\sinh t)^{n-1} dt \sim_n \frac{r^n}{1+r^n} e^{(n-1)r},
\end{equation}
where $\Omega_n$ is the euclidean $(n-1)$-volume of the sphere $S^{n-1}$, and the subindex in the symbol $\sim$ means that the constant behind this symbol depends only on the dimension $n$. 
This exponential behaviour, as well as the metric properties of $\cH^n$, make the classical covering arguments fail. In consequence, it is unclear how to decompose the level set $\left\{ x\in \H^n \,:\,Mf(x)>\lambda\right\}$ in such way that the appropriate averages of $w$ appear. 

\medskip

As in the euclidean case, from now on, given a non-negative measurable function $w$ (a weight) defined on $\H^n$, let $w(E)=\int_E w(x) d\mu_{n}(x)$  for a measurable set $E\subset \H^n$. On the other hand, given $s>1$, let 
$$
M_{s}w=M(w^{s})^{1/s}.
$$ 
Using this notation, our first main result is the following variant of the Fefferman-Stein inequality. 

\begin{thm} \label{F-S} For every weight $w\geq0$ 
we have that 
\[
w\left(\left\{ x\in \H^n \,:\,Mf(x)>\lambda\right\} \right)\leq C_{s,n}\frac{1}{\lambda}\int_{\H^n}|f(x)|M_{s}w(x)d\mu_{n}(x)
\]
where the constant $C_{s,n}\rightarrow+\infty$ when $s\rightarrow1$. 
\end{thm}  

\medskip

This theorem is a generalization of the result of Str\"omberg \cite{Str}, and as far as we know, it represents the first result for general weights in the hyperbolic setting. The reader may wonder if this result could hold for $s=1$. We will show that this result is false in general if $s=1$ (see Example \ref{Exa} item \ref{Ex1} below). Moreover, our example shows that it is false, even if we put iterations of the maximal function in the right hand side. In some sense, this is an evidence of a stronger singularity of the maximal function in the hyperbolic setting.  In Section 4  we will show that there are non trivial weights satisfying the pointwise condition $M_s(w)(x)\leq C w(x)$  a.e $x\in \H^n$. Then, for these weights it holds that the maximal function $M$ satisfies the weak type $(1,1)$ respect to the measure $wd\mu_{n}$. 

\bigskip

\subsubsection*{About the proof of Theorem \ref{F-S}} For each $r> 0$, let $A_{r}$ be the averaging
operator 
\[
A_{r}f(x)=\frac{1}{\mu_{n}(B_H(x,r))}\int_{B_H(x,r)}|f(y)|\,d\mu_n(y).
\]
Hence $Mf(x)=\sup_{r\ge0}A_{r}f(x)$. If $M^{loc}(f)$ denotes the operator obtained if supremum is restricted to $r \leq 2$, and $M^{far}(f)$ denotes the operator obtained if the supremum is taken over all $r \ge 2$, then  
$$
Mf(x)\leq M^{loc}f(x)+M^{far}f(x).
$$ 
On the one hand, the operator $M^{loc}$ behaves as in the Euclidean setting. The main difficulties appear in the estimations of $M^{far}$. In \cite{Str}, 
Str\"omberg uses a pointwise inequality obtained by Clerc and Stein in \cite{CS}. This pointwise inequality reduced the problem to get a good estimate for a convolution operator associated with a $k$-bi-invariant kernel $\tau$, which in the case of hyperbolic setting is $\tau(z,w)=(1+\mu_n(B(0,d(z,w))^{-1}$. A similar approach was used by Li and Lohou\'e in \cite{LiLo} to obtain sharp constants with respect to the dimension $n$. However, Str\"omberg's argument strongly uses the homogeneity of the measure $\mu_n$. So, it is not clear that one can apply a similar idea in the general case of any weight $w$. This makes it necessary to look for a more flexible approach.

\medskip

Our general strategy is based in the scheme used by Naor and Tao in \cite{NT}, where the weak type $(1,1)$ of the centered maximal function on the discrete setting of rooted $k$-ary trees is obtained. The flexibility of this approach was shown in \cite{ORS} and \cite{OR}, where the authors used this approach to get weighted estimates in the same discrete setting. It is well known that regular trees can be thought as discrete models of the hyperbolic space. Moreover, this kind of heuristic was used by Cowling, Meda and Setti in \cite{CMS}, but in the other way round, that is, in this work the authors used Str\"omberg's approach to prove weak estimates in the setting of trees. A novelty of our paper is to bring ideas of the discrete setting to the continue hyperbolic context. Adapting this strategy to a continuous context requires overcoming certain obstacles. On the one hand, the combinatorial arguments used in the discrete setting of trees are not longer available, so they have to be replaced by geometrical arguments. In this sense, the following estimate (Proposition \ref{intersection of balls})
$$
\mu_n\Big( \sub{B}{H}(y,s)\cap\sub{B}{H}(x,r)\Big)\leq C_n e^{\,\frac{n-1}{2}(\,r+s-d_{n}(x,y)\,)}
$$
is behind many estimates, as well as, some examples. It will also play a key role in the inequality
\[
\int_{F}A_r(\chi_E)(y) w(y) d\mu_n(y)\le c_{s,n}\ e^{-(n-1)\frac{r}{s'+1}}w(F)^{\frac{1}{s'+1}}M_{s}w(E)^{\frac{s'}{s'+1}},
\]
that is very important to prove Theorem \ref{F-S}. In this inequality, $E$ and $F$ are measurable subsets of $\H^n$, $s>1$, $s'=\frac{s}{s-1}$, and $r$ is a positive integer. On the other hand, in our setting the measure is not atomic. This leads us to make some estimations on some convenient averages of the original function instead of the function itself (see for instance Lemma \ref{lem:sumLevels}). 

\subsection{Weighted estimates in the hyperbolic space for $p>1$}
In the Euclidean case, the weak and strong boundedness of the maximal operator $M$ in weighted $L^p$ spaces is completely characterized by the $A_p$ condition defined in the seminal work of Muckenhoupt \cite{Mu}:
\begin{equation}\label{Ap Rn}
\sup \left(\frac{1}{|B|}\int_{B}w\, dx\right)\left(\frac{1}{|B|}\int_{B}w^{-\frac{1}{p-1}}\, dx\right)^{p-1}<\infty,
\end{equation}
where the supremum is taken over all the Euclidean balls. Different type of weighted inequalities were proved for measures such that the measure of the balls grows polynomically with respect to the radius (see for instance \cite{GCM}, \cite{NTV}, \cite{OP}, \cite{TTV}, and \cite{To}). However, the techniques used in those works can not be applied in our framework because of the geometric properties of $\cH^n$ and the exponential growth of the measures of balls with respect to the radius. Unweighted strong $(p,p)$ inequalities for the maximal function were proved for $p>1$ by Clerc and Stein in \cite{CS}. Moreover, singular integral operators also were studied on symmetric spaces by Ionescu (\cite{Io1,Io2}). 

\medskip

Roughly speaking, in the hyperbolic spaces, the behaviour of the maximal function is a kind of combination of what happens in the Euclidean case and in the trees. More precisely, recall that we have defined the operators
$$
M^{loc}f(x)=\sup_{0<r\leq 2}A_{r}f(x)\quad\mbox{and}\quad M^{far}f(x)=\sup_{2<r}A_{r}f(x).
$$
As we have already mentioned, the operator $M^{loc}$ behaves as if it were defined in the Euclidean space. So, it is natural to expect that it boundedness could be controlled by a kind of ``local $A_p$ condition''. We say that a weight  $w\in A_{p,loc}(\H^n)$ if 
\begin{equation*}
\sup_{0<r(B)\leq 1 }\left(\frac{1}{\mu_{n}(B)}\int_{B}w \mu_{n}\right)\left(\frac{1}{\mu_{n}(B)}\int_{B}w^{-\frac{1}{p-1}} \mu_{n}\right)^{p-1}<\infty.\label{eq:ApM}
\end{equation*}

\bigskip

The situation is very different for large values of the radius, when the hyperbolic structure comes into play. For instance, it is not difficult to show that the natural $A_p$ condition is too strong for the boundedness of $M^{far}$ in the hyperbolic setting. Indeed, in the Example \ref{Exa} we show a weight for which the maximal function is bounded in all the $L^p$-spaces, but it does not belong to any (hyperbolic) $A_p$ class. This suggests to follow a different approach. Inspired by the condition introduced in \cite{OR}, in the case of $k$-ary trees, we are able to define sufficient conditions to obtain weak and strong estimates for the maximal function respect to a weight $w$. Our main result in this direction is the following:

\begin{thm}\label{p>1}
\label{thm:Suff} Let $p>1$ and $w$ a weight. Suppose that
\begin{enumerate}
    \item[i.)]  $w\in A_{p,loc}(\H^n)$.
    \item[ii.)] There exist $0<\beta<1$ and $\beta\leq\alpha<p$ such that  for every $r\ge 1$ we have 
\begin{equation}
\int_{F} A_r(\chi_E)(y) w(y) d\mu_n(y) \lesssim e^{(n-1)r(\beta-1)}w(E)^{\frac{\alpha}{p}}w(F)^{1-\frac{\alpha}{p}},\label{SuffCond}
\end{equation}
for any pair of measurable subsets $E,F\subseteq \H^n$.
\end{enumerate}
Then
\begin{equation}
\|Mf\|_{L^{p,\infty}(w)}\lesssim\|f\|_{L^{p}(w)}.\label{eq:Weakbpalpha}
\end{equation}
Furthermore, if $\beta<\alpha$ then for each fixed $\gamma\ge 0$ we have 
\begin{equation}\label{sumAj}
\sum_{j=1}^{\infty}j^{\gamma}\|A_{j}f\|_{L^{p}(w)}\lesssim\|f\|_{L^{p}(w)}.
\end{equation}
And therefore
\begin{align*}
\|Mf\|_{L^{p}(w)}&\lesssim\|f\|_{L^{p}(w)}, \\  \|Mf\|_{L^{p'}(\sigma)}&\lesssim\|f\|_{L^{p'}(\sigma)},    
\end{align*}
where $\sigma=w^{{1-p'}}$ and $p'=\frac{p}{p-1}$.
\end{thm}

\begin{rem}\label{Sum_j}
 We observe that the estimate (\ref{sumAj}) in the previous theorem is stronger than the boundedness of the maximal function $M^{far}(f)$. In particular, it implies that if an operator $T$ satisfies the pointwise  estimate 
 $$|Tf(x)|\lesssim M^{loc}(|f|)(x)+\sum_{j\ge1} j^{\gamma} A_j(|f|)(x),$$ for some $\gamma\ge 0$, then the requested conditions on the weight $w$ in Theorem \ref{p>1} will be sufficient condition for the boundedness of $T$ in the space $L^p(w)$ with $p>1$. In particular, this generalized, in the hyperbolic setting, the unweighted estimates obtained by Clerc and Stein in \cite[Thm. 2]{CS} for the maximal function.
\end{rem}
\begin{rem}
It is not clear whether or not the condition (\ref{SuffCond}) for $\alpha=\beta$ is a necessary condition for the weak type $(p,p)$ boundedness of $M$ with respect to $w$. However, the condition is sharp in the following sense:  if $\beta=\alpha$ we can construct a weight for which the weak type $(p,p)$ holds, but the strong type $(p,p)$ fails. Consequently, the weak type $(q,q)$ fails as well for every $q<p$ (see Example \ref{Exa} (\ref{Ex2})). In particular, this shows that, unlike the classical case, in the hyperbolic context the weak $(p,p)$ inequality with respect to $w$ of the maximal operator is not equivalent to the strong estimate for $p>1$. 
\end{rem}
     
The condition (\ref{SuffCond}) could be not easy to be checked. For this reason, we consider the following result which provides a more tractable condition. To simplify the statement, given a positive integer $j$, let 
$$
\cC_{j}=B(0,j)\setminus B(0,j-1).
$$

Observe that the sets considered in the condition in \eqref{SuffCond} may have non-empty intersection with several different levels $\cC_j$. The condition in the following proposition studies the behavior of the weight at each level.

\begin{prop}\label{SuffP}
\label{cor:SuffAux}Let $1< p<\infty$, and let $w$ be a weight such that there exists a real number $\delta<1$,
so that for every $j,l,r\ge1$ integers with the restriction $|l-j|\leq r$, we have that
\begin{equation}
w(\cC_l\cap B(x,r))\lesssim e^{(n-1)\frac{r+l-j}{2}(p-\delta)}e^{(n-1)r\delta}w(x),\quad \mbox{for a.e.}\ x\in \cC_{j}.\label{Suffweaker}
\end{equation}
Then,  the condition \eqref{SuffCond} in Theorem \ref{p>1} holds with $\beta=\alpha=\frac{p}{p-\delta+1}$. 

\end{prop}

Combining Theorem \ref{p>1}, Remark \ref{Sum_j} and Proposition \ref{SuffP} we obtain the following corollary. 

\begin{cor}\label{CSU}
\label{cor:SuffAux}Let $1\leq p<\infty$, and $w\in A_{p,loc}(\H^n)$ such that there exists a real number $\delta<1$
such that for every $j,l,r\ge1$ integers with the restriction $|l-j|\leq r$, we have that
\begin{equation*}
w(\cC_l\cap B(x,r))\lesssim e^{(n-1)\frac{r+l-j}{2}(p-\delta)}e^{(n-1)r\delta}w(x),\quad \mbox{for a.e.}\ x\in \cC_{j}.
\end{equation*}
Then
\begin{equation*}
\|Mf\|_{L^{p,\infty}(w)}\lesssim\|f\|_{L^{p}(w)}.
\end{equation*}
Furthermore, if $p<q$ we have 
\begin{equation*}
\|Tf\|_{L^{q}(w)}\lesssim\|f\|_{L^{q}(w)},
\end{equation*}
for every operator $T$ satisfying the  pointwise  estimate 
$$
|Tf(x)|\lesssim M^{loc}(|f|)(x)+j^{\gamma}\sum_{j\ge1} A_j(|f|)(x),
$$
for some $\gamma\ge 0$.
\end{cor}

\subsection{Organization of the paper} This paper is organized as follow. In Section \ref{geo} we prove an estimate on the measure of the intersection of two hyperbolic balls. Section \ref{main} is devoted to the proof of the main results of this paper. The proof of Theorem \ref{F-S} is contained in Subsection \ref{proof F-S}, while the proof of Theorem \ref{p>1} is contained in Subsection \ref{proof of p>1}. The Section \ref{main} concludes with the proof of Proposition \ref{SuffP}.  The Section \ref{ejemplos varios}
contains examples that clarify several points previously mentioned. Finally, the paper concludes with an appendix on the ball model of the hyperbolic space.

\section{Geometric results}\label{geo}
\subsection{The hyperbolic space}

Although the precise realisation of hyperbolic space is not important for our purposes, for sake of concreteness, throughout this article we will consider the ball model. Recall that $\mu_n$ denotes the volume measure, and by $d_n$ we will denote the hyperbolic distance. A brief review of some basic facts about this model and its isometries is left to the Appendix \ref{ball model}.

\medskip

\subsection{Two results on the intersection of balls in the hyperbolic space}

This subsection is devoted to prove the following two geometric results, which will be very important in the sequel.

\begin{prop}\label{intersection of balls}
Let $\sub{B}{H}(y,s)$ and $\sub{B}{H}(x,r)$ be two balls in $\H_n$. Then
$$
\mu_n\Big( \sub{B}{H}(y,s)\cap\sub{B}{H}(x,r)\Big)\leq C_n e^{\,\frac{n-1}{2}(\,r+s-d_{n}(x,y)\,)},
$$
where $C_n$ is a constant that only depends on the dimension.
\end{prop}
\bdem
We can assume that $\sub{B}{H}(y,s)\cap\sub{B}{H}(x,r)\neq\varnothing$. On the other hand, since the estimate is trivial if $r$ and $s$ are less than a fixed constant, we can also assume that $r,s>2$. Without loss of generality, we can assume that $y=0$ and $x=(d,0,\ldots,0)$ with $d=d_n(x,y)$. Note that we can also assume that $d>0$, otherwise the estimate is trivial. The geodesic passing through the centers is the segment
$$
L=\{(t,0,\ldots,0):\ t\in (-1,1)\}.
$$
Since the balls are symmetric with respect to this geodesic line, the intersection is also symmetric with respect to this line. Let $O_L(n-1)$ be the subgroup of  the orthogonal group $O(n)$ defined by
$$
O_L(n)=\{A\in O(n): \mbox{A leaves invariant the geodesic line $L$}\},
$$
then the intersection is invariant by the action of $O_L(n-1)$. Moreover, the subgroup $O_L(n-1)$ acts transitively in the intersection of the boundaries $\partial\sub{B}{H}(0,s)\cap\partial\sub{B}{H}(x,r)$, which turns out to be an $(n-2)$-sphere. Let $S$ denote this intersection of boundaries, and consider the point $m\in L$ that satisfies
$$
d_n(0,m)=\frac{s+d-r}{2} \quad \Longleftrightarrow \quad d_n(m,x)=\frac{r+d-s}{2}.
$$
Since $L$ is a symmetry axis for $S$, the points in $S$ are at the same distance to the  point $m$. 
Let $\rho$ denote this distance. The volume of the ball of radius $\rho$ can be estimated using the hyperbolic law of cosines. Take $q\in S$, and consider the two dimensional hyperbolic (also linear) plane $P$ containing $q$ and $L$.  Let us restrict our attention to this hyperbolic plane (see Figure \ref{inter}).
\begin{figure}[H]
           \centering
           \includegraphics[height=4.5cm]{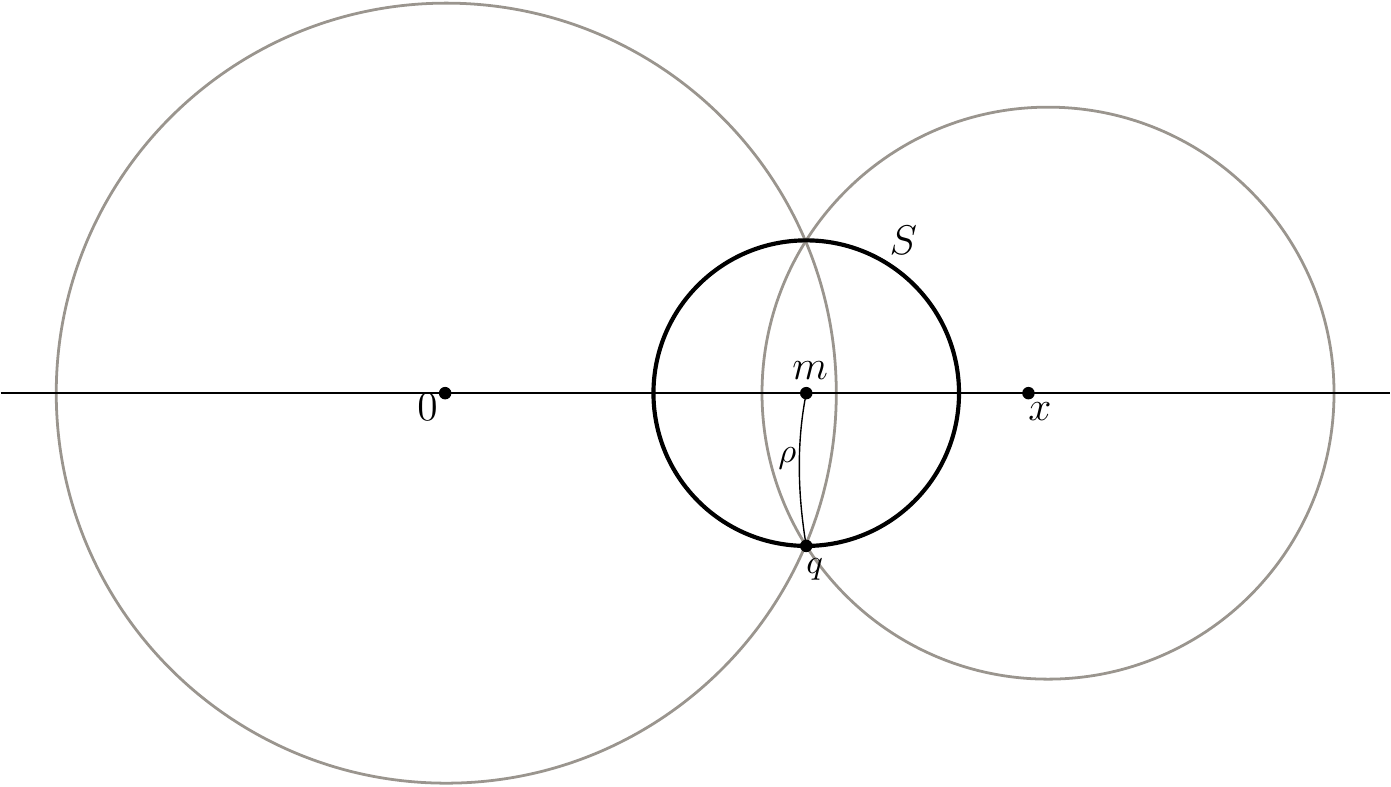}
           \par\vspace{0cm}
           \caption{Intersection of the balls with the two dimensional plane $P$.}
           \label{inter}
\end{figure} 
Since  $\angle(0,m,q)+\angle(q,m,x)=\pi$, one of them is greater or equal to $\frac{\pi}{2}$. Suppose that the angle $\theta=\angle(0,m,q)$ is greater than $\frac{\pi}{2}$, and consider the geodesic triangle whose vertices are $0$, $m$ and $q$ (see Figure \ref{triangulo})\footnote{If the angle $(0,m,q)$ were smaller than $\frac{\pi}{2}$, we use the angle $(q,m,x)$ and the triangle with vertices $q$, $m$ and $x$. }. 
\begin{figure}[H]
           \centering
           \includegraphics[height=3cm]{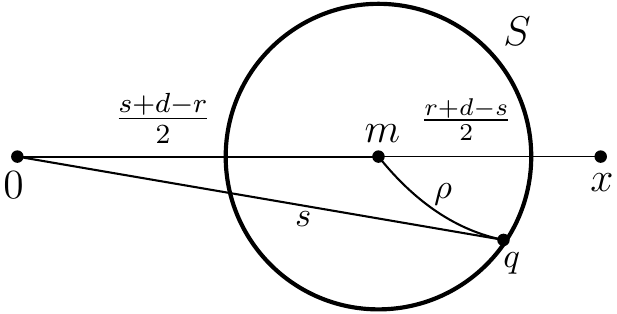}
           \par\vspace{0cm}
           \caption{Geodesic triangle.}
           \label{triangulo}
\end{figure} 
Since $\cos(\theta)$ is non-positive, we have that
\begin{align*}
\cosh(s)&= \cosh\Big(\frac{s+d-r}{2}\Big)\cosh (\rho)-\sinh\Big(\frac{s+d-r}{2}\Big)\sinh (\rho) \cos(\theta) \\
&\geq \cosh\Big(\frac{s+d-r}{2}\Big)\cosh (\rho).
\end{align*}

\medskip

Therefore, we get the following estimate
$$
e^\rho\leq \cosh(\rho) \leq \frac{\cosh(s)}{\cosh\Big(\frac{s+d-r}{2}\Big)}\leq 2 e^{\frac{s+r-d}{2}}.
$$
By equation \eqref{lanza la bola chico}, we get that
\begin{equation}\label{bola rho}
\Vol \Big(\sub{B}{H}(m,\rho)\Big)= \Omega_n \int_0^\rho (\sinh t)^{n-1} dr \leq K_n e^{(n-1)\rho} \leq 2^nK_n e^{(n-1)\left(\frac{s+r-d}{2}\right)}.
\end{equation}

\medskip

Now, it is enough to prove that $\sub{B}{H}(0,s)\cap\sub{B}{H}(x,r)\subseteq \sub{B}{H}(m,\rho)$. Since the intersection is an open-connected set, it is enough to prove that the boundary $\sub{B}{H}(m,\rho)$ is not contained in the intersection. So, take $p\in \partial\sub{B}{H}(m,\rho)$. By a continuity argument, we can assume that $p\notin L$. Then, as before, consider the plane $P$ generated by $p$ and the geodesic $L$. The geodesic $L$ divide this plane in two parts. Let $q$ be the unique point in $P\cap S$ in the same half-plane as $p$, and suppose that $\theta_p=\angle(p,m,x)$ is greater or equal than $\theta_q=\angle(q,m,x)$  (see Figure \ref{triangles}).
\begin{figure}[H]
           \centering
           \includegraphics[height=3.5cm]{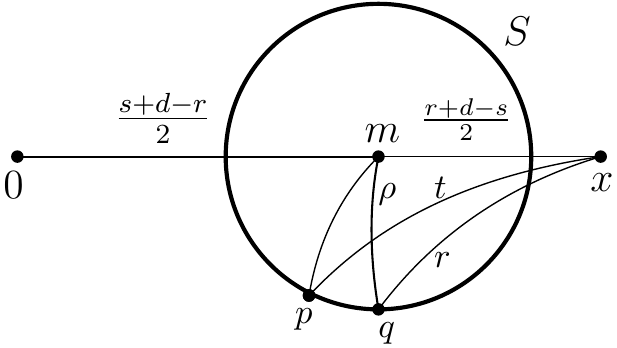}
           \par\vspace{0cm}
           \caption{Comparison of triangles.}
           \label{triangles}
\end{figure} 

If $t=d_n(x,p)$, since the cosine is decreasing in $(0,\pi)$ we get that
\begin{align*}
\cosh(t)&= \cosh\Big(\frac{r+d-s}{2}\Big)\cosh (\rho)-\sinh\Big(\frac{r+d-s}{2}\Big)\sinh (\rho) \cos(\theta_p) \\
&\geq  \cosh\Big(\frac{r+d-s}{2}\Big)\cosh (\rho)-\sinh\Big(\frac{r+d-s}{2}\Big)\sinh (\rho) \cos(\theta_q) \\
&=\cosh(r).
\end{align*}
In consequence, $t\geq r$ and therefore, the point $t\notin \sub{B}{H}(x,r)$. If $\angle(p,m,x)$ is smaller than $\angle(q,m,x)$, it holds that $\angle(p,m,0)$ is greater than $\angle(q,m,0)$. Hence, the same argument, replacing the vertex $x$ by the vertex $0$ shows that $t\notin  \sub{B}{H}(0,s)$. This concludes the proof.
\edem

\medskip

The following is a corollary of the proof of the previous lemma.

\begin{cor}\label{inclusion de bolas}
Let $\sub{B}{H}(0,s)$ and $\sub{B}{H}(x,r)$ be two balls in $\H_n$ such that their intersection has positive measure. If $\rho_0=\frac{1}{2}(\,r+s-d_{n}(0,x)\,)$, then
$$
\sub{B}{H}(m,\rho_0)\subseteq \sub{B}{H}(0,s)\cap\sub{B}{H}(x,r)\subseteq \sub{B}{H}(m,\rho_{0}+1),
$$
where $m=\alpha x$, and $\ds \alpha=\tanh\Big(\frac{s+d-r}{2}\Big)$.
\end{cor}

\section{Proof of Main results}\label{main}

First of all, we will prove the following arithmetical lemma, which is a slight generalization of a result  contained in \cite{OR}.

\begin{lem}\label{opti}

Let $1\le p<\infty$, $-p<\delta<1$,and  $\kappa>1$. Let the sequences  of non-negative real numbers
 $\{c_{j}\}_{j=0}^{\infty}$  and $\{d_{l}\}_{l=0}^{\infty}$ satisfying  
 
\begin{equation*}
\sum_{j=0}^{\infty}\kappa^{(p-\delta)j}c_{j}=A \quad\mathrm{and}\quad\sum_{l=0}^{\infty}\kappa^{l}d_{l}=B.\label{eq:redfcjdj}
\end{equation*}
 Then, for every integer $r\ge 1$ we have that

\begin{equation}
\sum_{\substack{j,l\in\N\cup\{0\}}}\min\left\{ \kappa^{\delta r}\kappa^{\frac{(l+j+r)(p-\delta)}{2}}c_{j},\kappa^{\frac{l+j+r}{2}}d_{l}\right\} \leq c_{p,\delta, \kappa}\, \kappa^{\frac{p}{p-\delta+1}r}  A^{\frac{1}{p-\delta+1}} B^{1-\frac{1}{p-\delta+1}}.
\end{equation}

\end{lem}

\begin{proof}
To prove this inequality, let $\rho$ be a real parameter to be
chosen later, and argue as follows 
\begin{eqnarray*}
 &  &\sum_{\substack{j,l \in\mathbb{N}\cup\{0\}
}
}\min\left\{ \kappa^{\delta r}\kappa^{\frac{(l+j+r)(p-\delta)}{2}}c_{j},\kappa^{\frac{l+j+r}{2}}d_{l}\right\} \\
 &  & \leq \kappa^{\frac{p+\delta}{2}r}\sum_{\substack{l,j\in\N\cup\{0\}\\
l<j+\rho
}
}\kappa^{\frac{(l+j)(p-\delta)}{2}}c_{j}+\kappa^{\frac{r}{2}}\sum_{\substack{l,j\in\N\cup\{0\}\\
l\geq j+\rho
}
}k^{\frac{l+j}{2}}d_{l}\\
 &  & \lesssim \kappa^{\frac{p+\delta}{2}r}\sum_{j=0}^{\infty}\kappa^{\frac{(j+\rho+j)(p-\delta)}{2}}c_{j}+\kappa^{\frac{r}{2}}\sum_{l=0}^{\infty}\kappa^{l-\frac{\rho}{2}}d_{l}\\
 &  & =\kappa^{\frac{p+\delta}{2}r}\kappa^{\frac{\rho(p-\delta)}{2}}\sum_{j=0}^{\infty}k^{j(p-\delta)}c_{j}+\kappa^{\frac{r}{2}}k^{-\frac{\rho}{2}}\sum_{l=0}^{\infty}\kappa^{l}d_{l}\\
 &  & =\kappa^{\frac{p+\delta}{2}r}\kappa^{\frac{\rho(p-\delta)}{2}}A+\kappa^{\frac{r}{2}}\kappa^{-\frac{\rho}{2}}B.
\end{eqnarray*}
Choosing
$\rho=\frac{2\log_{\kappa}\left(\frac{B}{A}\right)}{p-\delta+1}-\frac{(p+\delta-1)r}{p-\delta+1}$, it  follows that
\[\kappa^{\frac{p+\delta}{2}r}\kappa^{\frac{\rho(p-\delta)}{2}}A+\kappa^{\frac{r}{2}}\kappa^{-\frac{\rho}{2}}B\leq c_{p,\delta}\kappa^{\frac{p}{p-\delta+1}r}A^{\frac{1}{p-\delta+1}} B^{1-\frac{1}{p-\delta+1}}, \] 
which concludes the proof.
\end{proof}

\subsection{Proof of Theorem \ref{F-S}}\label{proof F-S}

The first step consists on proving that Lemma \ref{intersection of balls} leads to the following result. This is a key point to push the scheme on the discrete cases in \cite{NT} or \cite{ORS}. Recall that, given $r\geq0$, we denote by $A_{r}$ the averaging operator
\[
A_{r}f(x)=\frac{1}{\mu_n(B_H(x,r))}\int_{y\in B_H(x,r)}|f(x)|\,d\mu_n(x).
\]
\begin{lem}
\label{Borders}Let $E,F$  measurable sets of $\H^n$, $s>1$ and
let $r$ be a positive integer. Then

\[
\int_{F}A_r(\chi_E)(y) w(y) d\mu_n(y)\le c_{s,n}e^{-(n-1)\frac{r}{s'+1}}w(F)^{\frac{1}{s'+1}}M_{s}w(E)^{\frac{s'}{s'+1}},
\]

where $s'=\frac{s}{s-1}$ and $c_{s,n}$ is a constant depending on
$s$ and the dimension $n$.
\end{lem}

\begin{proof}
    
We divide the hyperbolic $\H^n$ in level sets as follows
$$
\H^n=\bigcup_{j=1}^{\infty}\cC_{j},
$$
where $\cC_j=\{x\in \H^n: j-1\leq \sub{d}{H}(0,x)<j \}$. Let $E_{j}=E\cap \cC_{j}$ and $F_{\ell}=F\cap \cC_{\ell}$.
Hence, we can write 

\begin{align}
I:=\int_{F}A_r(\chi_E)(y) w(y) d\mu_n(y)
 & = \ \ \,\sum_{\ell,j\geq 0}\ \ \,
 \int_{F_\ell}A_r(\chi_{E_j})(y) w(y) d\mu_n(y).
\label{eq:levels}
\end{align}

\medskip

Now, we will estimate the integrals
$$
I_{j,\ell}:=\int_{F_\ell}A_r(\chi_{E_j})(y) w(y) d\mu_n(y)
$$
in two different ways. On the one hand, given $x\in E_j$, let 
$$
\W_{j,\ell}^x=\{y\in F_\ell: \ d(x,y)\leq r\}.
$$ 
Then, by Lemma \ref{intersection of balls}

$$
\mu_n(\W_{j,\ell}^x)\leq C_n e^{\frac{n-1}{2}(\ell+r-j)}.
$$

\medskip

Using this estimate, we obtain that
\begin{align*}
I_{j,\ell}&=e^{-(n-1)r} \int_{F_\ell}\int_{B(y,r)}\chi_{E_j}(x)\,d\mu_n(x) w(y) d\mu_n(y)\\
&=e^{-(n-1)r} \int_{E_j} \int_{\W_{j,\ell}^x} w(y) d\mu_n(y)\, d\mu_n(x)\\
&=e^{-(n-1)r} \int_{E_j} \left(\int_{\W_{j,\ell}^x}\,d\mu_n\right)^{\frac{1}{s'}} \left(\int_{{B}_{H}(x,r)} w^s(y)\,d\mu_n(y)\right)^{\frac{1}{s}}\,d\mu_n(x)\\
&\leq C_n  e^{-(n-1)r} e^{\frac{n-1}{2s'}(\ell+r-j)}\, e^{\frac{(n-1)r}{s}}M_s(w)(E_j).
\end{align*}

\medskip
On the other hand, if $y\in F_\ell$, let $\W_{j,\ell}^y=\{x\in E_j: \ d(x,y)\leq r\}$. Then, by Lemma \ref{intersection of balls}

\begin{align*}
I_{j,\ell}&= e^{-(n-1)r}\int_{F_\ell} \int_{\W_{j,\ell}^y}  d\mu_n(x)\, w(y)d\mu_n(y)\\
&\leq C_n e^{-(n-1)r} e^{\frac{n-1}{2}(j+r-\ell)} \,w(F_\ell).
\end{align*}

In consequence
$$
I_{j,\ell}\leq C_n e^{-(n-1)r}\min\Big\{e^{\frac{n-1}{2s'}(\ell+r-j)}\, e^{\frac{(n-1)r}{s}}M_s(w)(E_j), e^{\frac{n-1}{2}(j+r-\ell)} \,w(F_\ell) \Big\},
$$
and
\begin{align*}
I&\leq C_n e^{-(n-1)r}\sum_{|\ell-j|\leq r+2} \min\Big\{e^{\frac{n-1}{2s'}(\ell+r-j)}\, e^{\frac{(n-1)r}{s}}M_s(w)(E_j), e^{\frac{n-1}{2}(j+r-\ell)} \,w(F_\ell) \Big\}.\\
\end{align*}

Now, if we define $c_{j}=\frac{M_{s}^\circ w(E_{j})}{e^{(n-1)\frac{j}{s'}}}$
and $d_{l}=\frac{w(F_{l})}{e^{(n-1)l}}$.  We have that 
\begin{equation}\label{series}
\sum_{j=0}^{\infty}e^{(n-1)\frac{j}{s'}}c_{j}={M_{s}^\circ}w(E)\quad\mathrm{and}\quad\sum_{j=0}^{\infty}e^{(n-1)l}d_{j}=w(F),
\end{equation}
and
\begin{align*}
\min&\Big\{e^{\frac{n-1}{2s'}(\ell+r-j)}\, e^{\frac{(n-1)r}{s}}M_s(w)(E_j), e^{\frac{n-1}{2}(j+r-\ell)} \,w(F_\ell) \Big\}\\
 &=\min\left\{ e^{\frac{(n-1)r}{s}}e^{(n-1)\frac{(l+j+r)}{2s'}}c_{j},e^{(n-1)\frac{l+j+r}{2}}d_{l}\right\}   
\end{align*}

Then we have that

\begin{equation} \label{Iest}
I\lesssim e^{-(n-1)r} \sum_{\substack{l,j\in\N\cup\{0\}}} \min\left\{ e^{\frac{(n-1)r}{s}}e^{(n-1)\frac{(l+j+r)}{2s'}}c_{j},e^{(n-1)\frac{l+j+r}{2}}d_{l}\right\}.  
\end{equation}
Now, if we choose $\delta=\frac{1}{s}$ and $p=1$ (then $p-\delta=\frac{1}{s'}$) we have that

$$
 \min\left\{ e^{\frac{(n-1)r}{s}}e^{(n-1)\frac{(l+j+r)}{2s'}}c_{j},e^{(n-1)\frac{l+j+r}{2}}d_{l}\right\} 
 $$
 is equal to
 $$
 \min\left\{ e^{(n-1)\delta r}e^{(n-1)\frac{(l+j+r)(p-\delta)}{2}}c_{j},e^{(n-1)\frac{l+j+r}{2}}d_{l}\right\}. 
$$

Therefore, if $\kappa=e^{n-1}$ and we take into account (\ref{series}), applying Lemma \ref{opti} in (\ref{Iest}) we get 

$$I\lesssim e^{-(n-1)\frac{r}{s'+1}}w(F)^{\frac{1}{s'+1}}M_{s}w(E)^{\frac{s'}{s'+1}}.$$ 
\end{proof}

We can use Lemma \ref{Borders} to obtain a distributional estimate on $A_{r}$. 
\begin{lem}
\label{lem:sumLevels} Let {$r\geq 1$} and $\lambda>0$. Then
\[
w\left(\left\{ A_{r} (A_1f) \geq\lambda\right\} \right)\lesssim c_{s}\sum_{k=0}^{r}\left(\frac{e^{(n-1)k}}{e^{(n-1)r}}\right)^{\frac{1}{2s'}}e^{(n-1)k}M_{s}w\left(\left\{ |A_2f|\geq \eta e^{(n-1)k}\right\} \right),
\]
where {$c_s$ depends only on $s$ and} $c_{s}\rightarrow\infty$ when $s\rightarrow1$.
\end{lem}

\begin{proof}[Proof of Lemma \ref{lem:sumLevels}]
Let $f_1=A_1 f$. We bound 
\begin{equation}
f_1\leq\frac{1}{e}+\sum_{k=0}^{r}
e^{(n-1)k}\chi_{E_{k}}+f_1\chi_{\{f_1\geq\frac{1}{2}e^{(n-1)r}\}},\label{eq:troceado}
\end{equation}
where $E_{k}$ is the sublevel set 
\begin{equation}
E_{k}=\left\{ e^{(n-1)(k-1)}\leq f_1<e^{(n-1)k}\right\} .\label{eq:niveles}
\end{equation}
Hence 
\begin{equation}
A_{r}f_1\leq\frac{1}{e}+\sum_{k=0}^{r}e^{(n-1)k}A_{r}\left(\chi_{E_{k}}\right)+A_{r}\left(f_1\chi_{\{f_1\geq\frac{1}{2}e^{(n-1)r}\}}\right).\label{eq:troceadopromediado}
\end{equation}
Given any $\la>0$ 
\begin{align*}
w\left(      \left\{ A_{r} \left(f_1\chi_{   \{f_1\geq e^{(n-1)r}\}  }\right)>\la   \right\}   \right)&\leq w\left(      \left\{ A_{r} \left(f_1\chi_{   \{f_1\geq e^{(n-1)r}\}  }\right)\neq 0   \right\}   \right)\\
&\leq w\left(      \left\{ x: \sub{B}{H}(r,x)\cap \{f_1\geq e^{(n-1)r}\}  \neq \varnothing   \right\}   \right).
\end{align*}

\medskip

Take $x$ such that $\sub{B}{H}(x,r)\cap \{f_1\geq e^{(n-1)r}\neq \varnothing$, and let $y$ be an element of this intersection. It is not difficult to see that
$$
 \sub{B}{H}(y,1)\subseteq \sub{B}{H}(x,r+1)\cap \big\{f_2\geq c e^{(n-1)r} \big\},
$$
where $f_2=A_2f$ and $c_0=\frac{\mu_n(B(0,1))}{\mu_n(B(0,2))}$. Therefore
\begin{align*}
w\left( \left\{ x: \sub{B}{H}(r,x)\cap \{f_1\geq e^{(n-1)r}\}  \neq \varnothing   \right\}   \right) &\leq w\Big( \Big\{A_{r+1} \left(\chi_{   \{f_2\geq c e^{(n-1)r}\}  }\right)>\frac{1}{c_1 e^{(n-1)r}}\Big\}\Big)\\
&\leq c_1 e^{(n-1)r} \int_{H_n} A_{r+1} \left(\chi_{   \{f_2\geq c e^{(n-1)r}\}  }\right) w d\mu \\
&\leq c_1 e^{(n-1)r} M(w) \left(\chi_{   \{f_2\geq c e^{(n-1)r}\}  }\right).
\end{align*}

On the other hand,
let $\beta\in (0,1)$ that will be chosen {later}. Note that if 
\[
\sum_{k=0}^{r}e^{(n-1)k}A_{r}\left(\chi_{E_{k}}\right)\geq\frac{1}{e},
\]
then we necessarily have {some} $k\in\N$ {such} that $1\leq k \leq r$ {for which}
\[
A_{r}\left(\chi_{E_{k}}\right)\geq\frac{e ^{(n-1)\beta}-1}{e^{(n-1)(k+2)}}\left(\frac{e^{(n-1)k}}{e^{(n-1)r}}\right)^{\beta}.
\]
Indeed, otherwise we have that
\begin{align*}
\frac{1}{e} & \leq\sum_{k=0}^{r}e^{(n-1)k}A_{r}\left(\chi_{E_{k}}\right)< 
\frac{e^{(n-1)\beta}-1}{e^{(n-1)(\beta r+2)}} 
\sum_{k=0}^{r} e^{(n-1)\beta k }\\
&=\frac{e^{(n-1)\beta}-1}{e^{(n-1)(\beta r+2)}}\  \frac{e^{(n-1)\beta(r+1)}-1}{e^{(n-1)\beta}-1}<\frac1e,
\end{align*}

which is a contradiction. Thus 
\begin{align*}
w\left(A_{r}f_1\geq{1}\right)\leq\sum_{k=0}^{r} w(F_{k})+ c_1 e^{(n-1)r} M(w) \left(\chi_{   \{f_2\geq c e^{(n-1)r}\}  }\right), 
\end{align*}
where 
\[
F_{k}=\left\{ A_{r}\left(\chi_{E_{k}}\right)\geq\frac{e ^{(n-1)\beta}-1}{e^{(n-1)(k+2)}}\left(\frac{e^{(n-1)k}}{e^{(n-1)r}}\right)^{\beta}\right\} .
\]
Note that $F_{k}$ has finite measure, and

\begin{align*}
w(F_k) \frac{e ^{(n-1)\beta}-1}{e^{(n-1)(k+2)}}\left(\frac{e^{(n-1)k}}{e^{(n-1)r}}\right)^{\beta}\leq \int_{F_k} A_r ( \chi_{E_k}) wd\mu_n(x).
\end{align*}

On the other hand, by Lemma \ref{Borders}{,}
\[
\int_{F_k} A_r ( \chi_{E_k}) wd\mu_n(x)\leq c_{s}e^{-(n-1)\frac{r}{s'+1}}w(F_{k})^{\frac{1}{s'+1}}M_{s}w(E_{k})^{\frac{s'}{s'+1}}.
\]
Hence
\[
w(F_{k})\frac{e ^{(n-1)\beta}-1}{e^{(n-1)(k+2)}}\left(\frac{e^{(n-1)k}}{e^{(n-1)r}}\right)^{\beta}\leq c_{s}e^{-(n-1)\frac{r}{s'+1}}w(F_{n})^{\frac{1}{s'+1}}{M_{s}}w(E_{n})^{\frac{s'}{s'+1}}.
\]

\medskip

So, choosing $\beta=\frac{1}{2(s'+1)}$ we have that
\begin{align*}
w(F_{k}) &\leq c_{s} e^{-(n-1)\frac{r}{2s'}}e^{\frac{(n-1)k}{2s'}}e^{(n-1)k}{M_{s}}w(E_{n})\\
        &\leq c_s\left(\frac{e^{(n-1)k}}{e^{(n-1)r}}\right)^{\frac{1}{2s'}}e^{(n-1)k}M_s w\left(\left\{f_1\geq e^{(n-1)(k-1)}\right\}\right).   
\end{align*}

Therefore
\begin{align}
  w(\{A_r f_1\geq 1\}) &\leq c_s\sum_{k=0}^{r} c_s\left(\frac{e^{(n-1)k}}{e^{(n-1)r}}\right)^{\frac{1}{2s'}}e^{(n-1)k}M_sw\left(\left\{f_1\geq e^{(n-1)(k-1)}\right\}\right)\nonumber\\
  &+ c_1 e^{(n-1)r} M(w) \left(\chi_{   \{f_2\geq c e^{(n-1)r}\}  }\right). \label{casi casi}
\end{align}
So, there exists $\eta>0$ depending only on the dimension such that
\begin{align*}
  w(\{A_r f_1\geq 1\}) &\leq \tilde{c}_s\sum_{k=0}^{r} \left(\frac{e^{(n-1)k}}{e^{(n-1)r}}\right)^{\frac{1}{2s'}}e^{(n-1)k}M_sw\left(\left\{f_2\geq \eta e^{(n-1)(k-1)}\right\}\right).
\end{align*}

Indeed, note that in the right-hand side of \eqref{casi casi}, the second term is dominated by the last term of the sum. This yields the desired conclusion.
\end{proof}

Combining the ingredients above we are in position to settle Theorem \ref{F-S}.

\begin{proof}[Proof of Theorem \ref{F-S}]
By the discussion in the introduction we only need to argue for $M^{far}(f)(x)$. Then, by Lemma~\ref{lem:sumLevels}
implies that 

\begin{align*}
w\Big(M^{far} f&\geq\lambda\Big)  \leq  w\left(M^{far} f_1\geq\lambda\right)\\
& \leq\sum_{r=1}^{\infty}w\left(A_{r} f_1\geq\lambda\right)\\
&=  \tilde{c}_s \sum_{r=0}^{\infty}\sum_{k=0}^{r} \left(\frac{e^{(n-1)k}}{e^{(n-1)r}}\right)^{\frac{1}{2s'}}e^{(n-1)k}M_sw\left(\left\{f_2\geq e^{(n-1)(k-1)}\eta \la\right\}\right)\\
&=  \tilde{c}_s\int_{\cH_n} \sum_{r=0}^{\infty}\sum_{k=0}^{r} \left(\frac{e^{(n-1)k}}{e^{(n-1)r}}\right)^{\frac{1}{2s'}}e^{(n-1)k} \chi_{\{f_2\geq e^{(n-1)(k-1)\}}\eta \la\}} M_sw(x)  d\mu_n(x)\\
&=  \tilde{c}_s\int_{\cH_n} \sum_{k=0}^{\infty}\sum_{r=k}^{\infty} \left(\frac{e^{(n-1)k}}{e^{(n-1)r}}\right)^{\frac{1}{2s'}}e^{(n-1)k} \chi_{\{f_2\geq e^{(n-1)(k-1)\}}\eta \la\}} M_sw(x)  d\mu_n(x)\\
&=  \tilde{c}_s\int_{\cH_n} \sum_{k=0}^{\infty}e^{(n-1)k} \chi_{\{f_2\geq e^{(n-1)(k-1)\}}\eta \la\}} M_sw(x)  d\mu_n(x)\\
&\leq \frac{\hat{c}_s}{\eta \la}\int_{\cH_n} f_2(x) M_sw(x)  d\mu_n(x)\\
&=\frac{\hat{c}_s}{\eta \la}\int_{\cH_n} f(x) A_2(M_sw)(x)  d\mu_n(x).
\end{align*}
Now, if $w$ is identically $1$ we have $A_2(M_sw)(x)=1$ and we are done. In particular, this recovers the Str\"omberg's weak type $(1,1)$ estimate. If $w$ is not constant, we claim that 
\begin{align*}
A_2\left((M_sw)\right)(x) \lesssim_{s} M_sw(x).
\end{align*}  
Indeed,
\begin{align*}
\frac{1}{\mu_n(B(x,2))} \int_{B(x,2)} M_sw(y) d\mu_n(y)&\leq
\frac{1}{\mu_n(B(x,2))} \int_{B(x,2)} M(w^s\chi_{B(x,4)}(y))^{\frac1s} d\mu_n(y) \\
&+ \frac{1}{\mu_n(B(x,2))} \int_{B(x,2)} M(w^s\chi_{(B(x,4))^c}(y))^{\frac1s} d\mu_n(y).
\end{align*}
The second term in the last line can be controlled by  $c M_s(w)(x)$ because  
$$
M(w^s\chi_{(B(x,4))^c}(y))^{\frac1s}\sim M(w^s\chi_{(B(x,4))^c}(x))^{\frac1s},
$$ 
for every $y\in B(x,2)$. Using Kolmogorov's inequality and the weak type $(1,1)$ of $M$ the first term can be estimate by $c_{\beta} (A_4(w^s)(x))^{\frac1s} $ and the claim follows. This completes the proof in the general case.  
\end{proof}

\bigskip

\subsection{Proof of Theorem \ref{p>1}}\label{proof of p>1}

The proof of Theorem \ref{p>1} follows the same ideas of the proof of Theorem 1.1. in \cite{OR}.  First, the hypothesis $w\in A_{p,loc}(\H^n)$ implies the estimates for $M^{loc}$ by standard arguments as in the classical setting. On the other hand, the arguments used to prove that Lemma \ref{Borders} implies Lemma \ref{lem:sumLevels} can be used to prove that the hypothesis in Theorem \ref{p>1} implies that 

\begin{equation}\label{redu}
w\left(\left\{ A_{r} (A_1f) \geq\lambda\right\} \right)\lesssim c_{s}\sum_{k=0}^{r}\left(\frac{e^{(n-1)k}}{e^{(n-1)r}}\right)^{\frac{1-\beta}{2}\frac{p}{\alpha}}e^{(n-1)\beta\frac{p}{\alpha} k}w\left(\left\{ |A_2f|\geq \eta e^{(n-1)k} \lambda\right\} \right).
\end{equation}

This inequality shows that the case $\beta<\alpha$ produces a better estimate than the case $\beta=\alpha$. First of all, assume that we are in the worst case $\beta=\alpha$. Arguing as in the proof of Theorem \ref{F-S} we get 

$$
w\left(\left\{ M^{far}f(x) \geq\lambda\right\} \right)\lesssim \frac{c}{\lambda^{p}} \int_{\H^n} |A_2(f)(x)|^{p} w(x)d\mu_n(x)dx. 
$$  

Since $|A_2(f)(x)|\leq M^{loc}f(x)$ and $w\in A_{p,loc}(\H^n)$, paying a constant we can eliminate $A_2$ in the right hand side of the previous estimate, and the proof is complete in this case. If we assume that $\beta<\alpha$, then by \eqref{redu} we have that

\begin{align*}
\|A_{r}f\|_{L^{p}(w)}^{p} & =p\int_{0}^{\infty}\lambda^{p-1}w\left(A_{r}f\geq\lambda\right)d\lambda\\
  & \lesssim \sum_{k=0}^{r} \left(\frac{e^{(n-1)k}}{e^{(n-1)r}}\right)^{\frac{1-\beta}{2}\frac{p}{\alpha}}e^{(n-1)\beta\frac{p}{\alpha} k}\int_{0}^{\infty}\lambda^{p-1}w\left(\left\{ |A_2f|\geq \eta e^{(n-1)k} \lambda\right\} \right)\\
 & =\sum_{k=0}^{r} \left(\frac{e^{(n-1)k}}{e^{(n-1)r}}\right)^{\frac{1-\beta}{2}\frac{p}{\alpha}}e^{(n-1)\beta\frac{p}{\alpha} k} e^{-(n-1)kp} \|A_2f\|^{p}_{L^{p}(w)}\\
& \lesssim e^{(n-1)rp(\frac{\beta}{\alpha}-1)} \|A_2f\|^{p}_{L^{p}(w)}.
\end{align*}
Since $w\in A_{p,loc}(\H^n)$ we can eliminate $A_2$ in the last norm, and taking into account that $\frac{\beta}{\alpha}-1<0$, we have that 
\[
\sum_{r=1}^{\infty} r^{\gamma} \|A_{r}f\|_{L^{p}(w)}\lesssim \sum_{r=1}^{\infty}  r^{\gamma} e^{(n-1)rp(\frac{\beta}{\alpha}-1)} \|f\|_{L^{p}(w)}\sim_{\gamma,\alpha,\beta, p} \|f\|_{L^{p}(w)}.
\]
This leads to (\ref{sumAj}). From (\ref{sumAj}) and the fact that $\sum_{j=1}^{\infty}A_j(f)$ is self-adjoint ($\gamma=0$) we obtain the boundedness of $M^{far}$ in the spaces $L^p(w)$ and $L^{p'}(\sigma)$. Moreover, since $w\in A_{p,loc}(\H^n)$ and therefore $\sigma$ is in $A_{p',loc}(\H^n)$ we have the same inequalities for $M^{loc}$, and as a consequence we obtain  
\begin{align*} \|Mf\|_{L^{p}(w)}&\lesssim\|f\|_{L^{p}(w)} \\  \|Mf\|_{L^{p'}(\sigma)}&\lesssim\|f\|_{L^{p'}(\sigma)}    
\end{align*}

This ends the proof of the Theorem.

\bigskip

\subsection{Proof of Proposition \ref{SuffP}}\label{facilongo}
The proof follows similar ideas as Lemma \ref{Borders}. 

\begin{proof}[Proof of Proposition \ref{SuffP}]   
Given $E,F$ subsets in $\H^n$, we should prove that
\begin{equation}\label{puntual1}
\int_{F} A_r(\chi_E)(y) w(y) d\mu_n(y) \lesssim e^{(n-1)r\left(\frac{p}{p-\delta+1}-1\right)} w(E)^{\frac{1}{p-\delta+1}}w(F)^{1-\frac{1}{p-\delta+1}}.
\end{equation}

Using the same notation as in the Lemma \ref{Borders}, we have

$$I_{j,\ell}:=\int_{F_\ell}A_r(\chi_{E_j})(y) w(y) d\mu_n(y).
$$
Given  $x\in E_j$, let $\W_{j,\ell}^x=\{y\in F_\ell: \ d(x,y)\leq r\}$. Then, by condition \eqref{Suffweaker}

$$
w(\W_{j,\ell}^x)\leq C_n e^{(n-1)\frac{r+l-j}{2}(p-\delta)}e^{(n-1)r\delta}w(x).
$$

Therefore,
\begin{align*}
I_{j,\ell}&=e^{-(n-1)r} \int_{F_\ell}\int_{B(y,r)}\chi_{E_j}(x)\,d\mu(x) w(y) d\mu_n(y)\\
&=e^{-(n-1)r} \int_{E_j} \int_{\W_{j,\ell}^x} w(y) d\mu_n(y)\, d\mu_n(x)\\
&\lesssim  e^{-(n-1)r} e^{(n-1)\frac{r+l-j}{2}(p-\delta)}e^{(n-1)r\delta}w(E_j).
\end{align*}

\medskip
On the other hand, if $y\in F_\ell$, let $\W_{j,\ell}^y=\{x\in E_j: \ d(x,y)\leq r\}$. Then, by Lemma \ref{intersection of balls}

\begin{align*}
I_{j,\ell}&= e^{-(n-1)r}\int_{F_\ell} \int_{\W_{j,\ell}^y}  d\mu_n(x)\, w(y)d\mu_n(y)\\
&\leq C_n e^{-(n-1)r} e^{\frac{n-1}{2}(j+r-\ell)} \,w(F_\ell).
\end{align*}
So, 
$$
I_{j,\ell}\leq C_n e^{-(n-1)r}\min\Big\{e^{(n-1)\frac{r+l-j}{2}(p-\delta)}e^{(n-1)r\delta}w(E_j), e^{\frac{n-1}{2}(j+r-\ell)} \,w(F_\ell) \Big\}.
$$

From now on, we can follow the same steps as in the proof of Lemma 3.2, and using Lemma 3.1 we obtain (\ref{puntual1}).

\end{proof}

\section{Examples}\label{ejemplos varios}

In this last section we show several examples to clarify several points previously mentioned. We omit details since the examples follow from  continue variants of Theorem 1.3 in \cite{OR}.

Let $-\infty<\gamma\leq1$, we denote  
$$
w_{\gamma}(x)=\frac{1}{\big(\,1+\mu_n(\,B(0,d_{H}(0,x)\,)\,\big)^{\gamma}}.
$$ 

\begin{exas}\label{Examples} $\ $

\begin{enumerate}\label{Exa}

\item \label{Ex1} If $0\leq\gamma\leq 1$, then $$M(w_\gamma)(x)\lesssim w_\gamma(x) $$ 

\noindent In particular if $\gamma<1$ taking $s>1$ such that $\gamma s \leq 1$ we have that  $$M_{s}(w_{\gamma})(x)\lesssim  w_{\gamma}(x)$$

\noindent Therefore there are non trivial weights satisfying $M_s(w)\lesssim w$. On the other hand, $Mw_1(x)\lesssim w_1(x)$. However, the weak type $(1,1)$ of $M$ with respect to $w_1$ fails.
In fact, taking $f_k(x)=\chi_{\cC_{k}}(x)$ for $k$ big, it is not difficult to show that $w_{1}\{x: M(f_k)(x)>1/2 \} \ge k$ and the $L^1(w_1)$-norm of $f_k$ is uniformly bounded. In particular, this example shows that in Theorem \ref{F-S} is not possible to put $s=1$. In fact, it is not possible to put any iteration ($M^m(f)=M(M^{m-1}f)$) of $M$ for any fixed natural number $m$.

\item\label{Ex2} Let $p>1$. Then $w_{1-p}(x)$ satisfies the hypothesis of Corollary  \ref{CSU}   and therefore
\[
\|Mf\|_{L^{p,\infty}(w_{1-p})}\lesssim\|f\|_{L^{p}(w_{1-p})}
\]
holds. Nevertheless,  $\|Mf\|_{L^{p}(w_{1-p})}\lesssim\|f\|_{L^{p}(w_{1-p})}$ does not. This can be seen by considering the function $f=\chi_{B(0,1)}$, and taking into account that $w\simeq (M\chi_{B(0,1)})^{1-p}$.

\item \label{Ex3} Fixed $\gamma\in\left(0,1\right)$. We have seen in the item 1 that the maximal function satisfies a weak type $(1,1)$ inequality for this weight. In particular, for every $q>1$,
\[
\|Mf\|_{L^{q}(w_{\gamma})}\lesssim\|f\|_{L^{q}(w_{\gamma})}.
\]
However, it is not difficult to see that, for any fixed $p>1$, it holds that
\[
\sup_{r>0}\frac{1}{\mu_{n}(B(0,r))}\int_{B(0,r)}w_{\gamma}\left(\frac{1}{\mu_{n}(B(0,r))}\int_{B(0,r)}w_{\gamma}^{-\frac{1}{p-1}}\right)^{p-1}=\infty.
\]
This example shows that boundedness of $M$ does not imply the natural condition $A_p$ for any $p>1$ in this setting. In the Euclidean setting in the context of a general measure $\mu$ an example in this line was also obtained by Lerner in \cite{L}.  
\end{enumerate}
\end{exas}

\appendix{

\section{The ball model of the hyperbolic space}\label{ball model}

 Let $\bm{n}=\{x\in\R^n:\ \|x\|<1\}$, where $\|\cdot\|$ denotes the euclidean norm in $\R^n$. In this ball we will consider the following Riemannian structure
$$
ds_x^2(v)=\frac{2\|v\|^2}{(1-\|x\|^2)^2}.
$$
The hyperbolic distance in this model can be computed by 
$$
d_n(x,y)=\arctanh \left(\frac{\|x-y\|}{(1-2\pint{x,y}+\|x\|^2\|y\|^2)^{\frac12}}\right).
$$
\medskip

The group of isometries $\isob{n}$ in this representation coincides with the group of conformal diffeomorphisms from $\bm{n}$ onto itself. For $n=2$, we can identify $\R^2$ with $\C$, and this group is the one generated by:

\begin{itemize}
\item Rotations: $\ds z\mapsto e^{it}z$, $t\in\R$.
\item M\"obius maps: $\ds z\mapsto \frac{z-w}{1-\bar{w}z}$.
\item Conjugation: $\ds z\mapsto \overline{z}$.
\end{itemize}

For dimension $n>2$, recall that, by Liouville's theorem, 
every conformal map between two domains of $\R^n$ has the form
$$
x\mapsto \la A\circ \iota_{x_0,\alpha} (x) +b
$$
where $\la>0$, $b\in\R^n$, $A$ belongs to the orthogonal group $O(n)$, and for $x_0\in\R^n$, $\alpha\in \R$ 
$$
\iota_{x_0,\alpha}(x)=\alpha \frac{x-x_0}{\|x-x_0\|^2}+x_0.
$$
Note that, when $\alpha>0$, the maps $\iota_{x_0,\alpha}$ correspond to a reflection with respect to the sphere 
$$
S^{n-1}(x_0,\alpha)=\{x\in\R^n:\ \|x-x_0\|^2=\alpha\}.
$$
If $\alpha<0$, it is a composition of the inversion with respect to the sphere $S^{n-1}(x_0,-\alpha)$ and the symmetry centered at $x_0$. Using this result, we get that the group $\isob{n}$ consists of the maps
of the form
$$
A\circ \theta
$$
where $A$ belongs to the orthogonal group $O(n)$ and  $\theta$ is either the identity or an inversion with respect to a sphere that intersect orthogonally $\partial \bm{n}$. Recall that we say that two spheres $S_1$ and $S_2$ intersects orthogonally if for every $p\in S_1\cap S_2$  
$$
(T_p S_1)^\bot\, \bot\  (T_p S_2)^\bot.
$$ 
\begin{rem}
This representation is also true for $n=2$. Indeed, on the one hand, the rotations as well as the conjugation belongs to $O(2)$. On the other hand, given $\alpha\in \C$ such that $|\alpha|<1$,  the circle of center $\alpha^{-1}$ and squared radius $|\alpha|^{-2}-1$ is orthogonal to $\partial \bm{2}$, and if $\iota$ denotes the inversion with respect to this circle then
$$
\iota(z)=\frac{\overline{z}-w}{1-\bar{w}\overline{z}}.
$$ 
\end{rem}

In this model, the $r$-dimensional hyperbolic subspaces that contains the origin are precisely the intersection the $r$-dimensional linear subspaces of $\R^d$ with $\bm{n}$. The other ones, are images of these ones by isometries. So, they are $r$-dimensional spheres orthogonal to $\partial \bm{n}$. The orthogonality in this case, as before, is defined in the natural way in terms of the orthogonal complements of the corresponding tangent spaces. 

}


\end{document}